\documentclass[a4paper,11pt]{article}

\usepackage{amsthm}
\usepackage{amssymb}
\usepackage{amsmath}

\usepackage{mathtools}
\usepackage{enumitem}
\usepackage{multirow}
\usepackage[normalem]{ulem}

\usepackage{tikz}
\usetikzlibrary{arrows,angles}

\theoremstyle{definition} 

\newcommand{\thistheoremnam}{}
\newtheorem*{genericthm*}{\thistheoremnam}
\newenvironment{chapt*}[1]
 {\renewcommand{\thistheoremnam}{#1}%
 \begin{genericthm*}}
 {\end{genericthm*}}

\newtheorem{thm}{Theorem}[section]
\newtheorem{pop}[thm]{Proposition}

\newtheorem{defi}[thm]{Definition}

\newtheorem{manualtheoreminner}{Theorem}
\newenvironment{introtheorem}[1]{%
 \renewcommand\themanualtheoreminner{#1}%
 \manualtheoreminner
}{\endmanualtheoreminner}

\theoremstyle{definition}
\newtheorem{rem}[thm]{Remark}
\newtheorem{ex}[thm]{Example}
\numberwithin{equation}{section}

\AtBeginEnvironment{ex}{%
 \pushQED{\qed}%
}
\AtEndEnvironment{ex}{\popQED\endex}

\newcommand{\LpLS}{Lorentzian pre-length space }
\newcommand{\LpLSn}{Lorentzian pre-length space}

\newcommand{\mb}[1]{\mathbb{#1}}

\usepackage{accents}

\usepackage{hyperref}

\renewcommand{\labelenumi}{(\roman{enumi})}
\renewcommand\theenumi\labelenumi

\title{Characterizing intrinsic Lorentzian length spaces via $\tau$-midpoints}
\author{Tobias Beran\thanks{{\tt tobias.beran@univie.ac.at}, Faculty of Mathematics, University of Vienna, Austria.}\ \ and Felix Rott\thanks{{\tt felix.rott@univie.ac.at}, Faculty of Mathematics, University of Vienna, Austria.}}
\date{}

\begin{document}
\maketitle
 
\begin{abstract}
In metric geometry, the question of whether a distance metric is given by the length of curves can be decided via the existence of midpoints with respect to the metric $d$. We adapt a similar characterization to the setting of Lorentzian pre-length spaces. In particular, we show that a given space is strictly intrinsic provided it has $\tau$-midpoints and merely intrinsic provided it has approximate $\tau$-midpoints. Our approach is based on the null distance of C.\ Sormani and C.\ Vega. 

\bigskip

\end{abstract}

\tableofcontents

\section{Introduction}

Metric geometry and the theory of Alexandrov spaces is a synthetic and axiomatic approach to Riemannian geometry, pioneered by \cite{Ale57}. The main object of study is a so-called \emph{length space}, a metric space where the distance is given by the infimum of lengths of connecting curves. 
The fundamental geometric concepts of Riemannian manifolds, such as angles, curvature bounds and tangent spaces are abstracted to the setting of metric spaces, and in this way open up a far-reaching generalization of to Riemannian geometry. 
\medskip

In the same spirit, the theory of Lorentzian length spaces is a synthetic approach to Lorentzian geometry. This is a rather new research area introduced in \cite{KS18}. The advancements in the synthetic Lorentzian setting can be roughly traced back to two different motivations (with some overlap). 
On the one hand, there is the attempt to solidify Lorentzian length spaces as a synthetic setting in its own right, i.e., translating concepts from metric geometry to Lorentzian length spaces. These include, but are of course not limited to, timelike curvature bounds in the sense of triangle comparison (generalizations of sectional curvature bounds), angles and (timelike) tangent cones, gluing techniques and other fundamental constructions, Hausdorff measure, and synthetic Ricci curvature bounds, cf.\ \cite{KS18, BS22, BOS22, BR22, Rot22, BORS23, BNR23, MS22, CM20,BKR23}. 
On the other hand, there is the goal of adapting useful concepts from the theory of relativity, or concepts and behaviours from Lorentzian geometry which might not be present in the Riemannian counterpart, to the synthetic setting. Examples include causality theory, (in)extendibility of spaces, warped products, time functions and the null distance, cf.\ \cite{ACS20, GKS19, AGKS19, BGH21, KS22}. 
\medskip

This work can be placed in the first category, as we are aiming to adapt a characterization of intrinsic distance metrics. In general, it is complicated to verify whether a given metric is given via the length of curves. 
Fortunately, there are numerous characterizations which are easier to check and still ensure that a metric is intrinsic, see e.g.\ \cite{BBI01,Pap14}. In particular, there is the well known characterization using so-called midpoints: a complete metric space is strictly intrinsic (or geodesic), i.e., there exists a shortest curve between any two points, if and only if there exists a midpoint between any pairs of points, i.e., a point which has half the total distance from both points, so visually it is situated exactly in the middle between these two points (precise definitions follow in the next section). 
For proper metric spaces, this notion is equivalent to Menger-convexity, cf.\ \cite[Theorem 2.6.2]{Pap14}. 
An analogous formulation also exists for merely intrinsic metrics and the existence of $\varepsilon$-midpoints, which can morally be described as ``almost-midpoints''. As an example, the fact that the Gromov-Hausdorff limit of length spaces is again a length spaces is very easy using this characterization, cf.\ \cite[Theorem 7.5.1]{BBI01}. 
\medskip

The proof of this equivalent description in the metric case is not difficult. 
However, the translation to the Lorentzian setting is far from straightforward, since, roughly speaking, the distance metric \emph{induces} the topology while the time separation function is in comparison barely involved with the topology. 
To overcome this, we rely on some compatibility with a given time function (more on this later). We now mention the main results of this work.
\begin{introtheorem}{\ref{thm: midpoints imply strictly intrinsic}}[$\tau$-midpoints imply being causally geodesic]
Let $X$ be a locally causally closed and sufficiently causally connected \LpLSn. Let $T: X \to \mb{R}$ be a locally anti-Lipschitz time function and suppose that $X$ is equipped with the null-distance $d_T$ with respect to $T$. 
Assume that $(X,d_T)$ is complete and that $X$ has $\tau$-midpoints. 
Assume further that there exists $c \in (0,\frac{1}{2}]$ such that the time function is compatible with midpoints in the following sense: for any $p, q \in X, p \ll q$ there is a $\tau$-midpoint $m$ of $p$ and $q$ such that 
\begin{equation*}
(1-c) T(p) + c T(q) \leq T(m) \leq c T(p) + (1-c) T(q) \, . 
\end{equation*}
Then $X$ is causally geodesic.
\end{introtheorem}
\begin{introtheorem}{\ref{thm: approx midpoints imply intrinsic}}[Approximate $\tau$-midpoints imply being intrinsic]
Let $X$ be a locally causally closed and sufficiently causally connected \LpLSn. Let $T: X \to \mb{R}$ be a time function and suppose that $X$ is equipped with the null-distance $d_T$ with respect to $T$. Assume that $(X,d_T)$ is complete and that $X$ has approximate $\tau$-midpoints. Assume further that there exists $c \in (0,\frac{1}{2}]$ and $\hat{\varepsilon}>0$ such that the time function is compatible with approximate $\tau$-midpoints in the following sense: for any $p, q \in X, p \ll q$ there is a $\tau(p,q)\hat{\varepsilon}$-$\tau$-midpoint $m$ of $p$ and $q$ such that 
\begin{equation*}
(1-c) T(p) + c T(q) \leq T(m) \leq c T(p) + (1-c) T(q) \, .
\end{equation*}
Then $X$ is intrinsic.
\end{introtheorem}

\section{Preliminaries}
In this section, we gather some fundamental background knowledge used to build up the main results. 
More precisely, we recall some properties of Lorentzian pre-length spaces, the null distance, and we repeat the metric characterization of being intrinsic via midpoints. 
We assume the reader to be familiar with the very basic concepts concerning the theory of Lorentzian length spaces, referring to \cite{KS18} for details. 
To begin with, recall that a \LpLS $X$ is called \emph{chronological} if $\tau(x,x)=0$ for all $x \in X$. 
Moreover, $X$ is called \emph{intrinsic} if $\tau$ is given by the supremum of lengths of causal curves, and \emph{strictly intrinsic} or \emph{geodesic} if for all timelike related pairs of points there exists a distance realizer between them. 
If there also exists a causal curve between any pair of null related points (which then necessarily is null and realizes the distance), then $X$ is called \emph{causally strictly intrinsic} or \emph{causally geodesic}. In particular, a causally path-connected and geodesic space is causally geodesic. 
\begin{defi}[Local causal closure]
 A \LpLS $X$ is called locally causally closed if every point has a neighbourhood $U$ such that $\leq|_{\bar{U} \times \bar{U}}$ is closed, i.e., if $p_n \to p, q_n \to q, p_n, q_n \in U, p, q \in \bar{U}, p_n \leq q_n$, then $p \leq q$.
\end{defi}
Note that there is the closely related notion of being \emph{locally weakly causally closed} introduced in \cite[Definition 2.19]{ACS20}. Under strong causality, these are equivalent, cf.\ \cite[Proposition 2.21]{ACS20}. 
At first glance this version seems more adequate for our setting, as we deal with intrinsic (or at least causally path-connected) spaces. However, since we do not assume strong causality, this property is not suitable for our purposes, see Remark \ref{rem: on loc caus closure}. 
Next, we give the definition of $\tau$-midpoints and approximate $\tau$-midpoints. 
\begin{defi}[(Approximate) $\tau$-midpoints]
 Let $X$ be a \LpLS and let $x \ll z$. 
 \begin{itemize}
 \item[(i)] $y \in X$ is called \emph{$\tau$-midpoint} of $x$ and $z$ if $\tau(x,y)=\frac{1}{2}\tau(x,z)=\tau(y,z)$. 
 \item[(ii)] Let $\varepsilon>0$. $y \in X$ is called \emph{$\varepsilon$-$\tau$-midpoint} of $x$ and $z$ if 
 \begin{equation}
 \label{eq: epsilon-midpoints definition}
 |\tau(x,y) - \frac{1}{2}\tau(x,z)|<\varepsilon \quad \text{and} \quad |\tau(y,z) - \frac{1}{2}\tau(x,z)|<\varepsilon \, . 
 \end{equation}
 
 \item[(iii)] $X$ is said to have \emph{$\tau$-midpoints} if any pair of timelike related points has a $\tau$-midpoint and $X$ is said to have \emph{approximate $\tau$-midpoints} if any pair of timelike related points has an $\varepsilon$-$\tau$-midpoint for all $\varepsilon>0$.
 \end{itemize}
\end{defi}
In Lorentzian geometry, it has always been difficult to relate the undyerlying topology of the manifold to the Lorentzian metric. This is in contrast to Riemannian geometry, where it is well known that the infimum of lengths of connecting geodesics induces a metric in the topological sense, which even induces the manifold topology. The null distance, due to C.\ Sormani and C.\ Vega, is a rather recent attempt of obtaining a similar interplay for Lorentzian manifolds using time functions and piecewise causal curves with alternating time orientation. 
It was first introduced in \cite{SV16} and subsequently adapted to the synthetic setting in \cite{KS22}. 
\begin{defi}[Time functions and being locally anti-Lipschitz]
 Let $X$ be a \LpLSn. 
 \begin{itemize}
 \item[(i)] A function $T : X \to \mb{R}$ is called \emph{generalized time function} if $x < y \Rightarrow f(x) < f(y)$.
 \item[(ii)] A continuous generalized time function is called a \emph{time function}.
 \item[(iii)] A generalized time function $T:X \to \mb{R}$ is called \emph{locally anti-Lipschitz} if every point has a neighbourhood $U$ together with a metric $d_U : U \times U \to \mb{R}$ such that for all $x,y \in U$ with $x \leq y$ we have $T(y) - T(x) \leq d_U(x,y)$.
 \end{itemize}
\end{defi}
\begin{defi}[Piecewise causal curves and the null distance]
 Let $X$ be a \LpLS and let $T$ be a generalized time function. 
 \begin{itemize}
 \item[(i)] A curve $\gamma : [a,b] \to X$ is called \emph{piecewise causal} if there exists a partition $a=t_0 < t_1 < \ldots < t_n = b$ such that $\gamma|_{[t_i, t_{i+1}]}$ is a future-directed or past-directed causal curve for all $i$.
 \item[(ii)] The \emph{null length} of a piecewise causal curve $\gamma$ is defined as 
 $L_{T}(\gamma) \coloneqq \sum_{i=0}^{n-1} |T(\gamma(t_{i+1}) - T(\gamma(t_i))|$.
 \item[(iii)] The \emph{null distance} between $p,q \in X$ is defined as 
 \begin{equation}
 \label{eq: null distance}
 d_T(p,q) \coloneqq \inf \{ L_T(\gamma) \mid \text{$\gamma$ is piecewise causal from $p$ to $q$}\} \, .
 \end{equation}
 \end{itemize}
\end{defi}
In order to have a reasonably behaved null distance, we have to assume that $X$ possesses some elementary properties.
\begin{defi}[Sufficient causal connectedness]
 A \LpLS $X$ is said to be \emph{sufficiently causally connected} if it is path-connected, causally path-connected and every point lies on a timelike curve.
\end{defi}
It then follows, cf.\ \cite[Lemma 3.5]{KS22}, that any two points can be joined by a piecewise causal curve. In particular, the infimum in \eqref{eq: null distance} is never empty and we also have 
\begin{equation}
 \label{eq: null distance between causal points}
 p \leq q \Rightarrow d_{T}(p,q)=T(q)-T(p) \, , 
\end{equation}
cf.\ \cite[Proposition 3.8(ii)]{KS22}.

The following properties of the null distance are crucial to our approach. They essentially guarantee that, under reasonable conditions, we can assume any \LpLS to be equipped with the null distance instead of an arbitrary metric, without losing any generality. 

\begin{thm}[Crucial properties of the null distance]
Let $X$ be a sufficiently causally connected and locally compact Lorentzian pre-length space equipped with a locally anti-Lipschitz time function $T$. 
Then $(X,d_T)$ is a metric space which is homeomorphic to $(X,d)$. 
\end{thm}
\begin{proof}
 See \cite[Propositions 3.9, 3.10 and 3.12]{KS22}.
\end{proof}
Recall that the metric distance on $X$ is only used as a background tool (replacing the role of any Riemannian background metric on a Lorentzian manifold), hence the exact values of $d$ are not important, it is mostly the topological information that is relevant. 
\medskip

Finally, we gather the necessary vocabulary and statements from metric geometry. For more details, we refer to \cite{BBI01, BH99, AKP19, Pap14}. 
Even though the rest of this section is well-known, for the convenience of the reader we briefly recall the arguments of the metric result concerning the equivalence of being intrinsic and the existence of midpoints. 
\begin{defi}[Midpoints and approximate midpoints]
 Let $(X,d)$ be a metric space and let $x,z \in X$.
 \item[(i)] $y \in X$ is called \emph{midpoint} of $x$ and $z$ if $d(x,y)=\frac{1}{2}d(x,z)=d(y,z)$.
 \item[(ii)] Let $\varepsilon>0$. $y \in X$ is called \emph{$\varepsilon$-midpoint} of $x$ and $z$ if $|d(x,y)-\frac{1}{2}d(x,z)|<\varepsilon$ and $|d(y,z)-\frac{1}{2}d(x,z)|<\varepsilon$. 
 \item[(iii)] $X$ is said to have \emph{midpoints} if for all $x,z \in X$ there exists a midpoint. $X$ is said to have \emph{approximate midpoints} if for all $x,z \in X$ and for all $\varepsilon>0$ there exists an $\varepsilon$-midpoint. 
\end{defi}

\begin{pop}[Intrinsic\footnote{We are never interested in whether the background metric $d$ on a \LpLS is intrinsic, so there is no danger of confusion. } metric has midpoints]
\label{pop: intrinsic has midpoints}
 Let $(X,d)$ be a metric space.
 \begin{itemize}
 \item[(i)] If $X$ is strictly intrinsic, then $X$ has midpoints.
 \item[(ii)] If $X$ is intrinsic, then $X$ has approximate midpoints.
 \end{itemize}
\end{pop}
\begin{proof}
 See \cite[Lemmas 2.4.8 and 2.4.10]{BBI01}. 
\end{proof}

The proof of the other direction relies on the following fundamental result about extending uniformly continuous functions, which will also be used for the Lorentzian proof.
\begin{thm}[Extending uniformly continuous functions]
\label{thm: extension theorem}
Let $X$ and $Y$ be metric spaces and assume that $Y$ is complete. Let $A \subseteq X$ be dense and let $f:A \to Y$ be a uniformly continous function. Then there exists a unique continuous extension of $f$ to all of $X$.
\end{thm}
\begin{proof}
 See e.g.\ \cite[Theorem 39.10]{Wil04} (recall that any metric space is a uniform space).
\end{proof}

\begin{pop}[Existence of midpoints implies being intrinsic]
 Let $(X,d)$ be a complete metric space.
 \begin{itemize}
 \item[(i)] If $X$ has midpoints, then $X$ is strictly intrinsic.
 \item[(ii)] If $X$ has approximate midpoints, then $X$ is intrinsic.
 \end{itemize}
\end{pop}
\begin{proof}[Proof (Rough idea)]
 We give a very simple proof idea for (i), following \cite[Theoreom 2.4.16]{BBI01}. 
 Let $x,z \in X$ and iteratively define a ``curve'' as follows: set $\gamma(0)=x, \gamma(1)=z$ and $\gamma(\frac{1}{2})=y$, where $y$ is a midpoint of $x$ and $z$. 
 Then define $\gamma(\frac{1}{4})$ to be a midpoint between $x$ and $y$ and define $\gamma(\frac{3}{4})$ to be a midpoint between $y$ and $z$.
 We can repeat this process to finally end up with a map from the dyadic rationals in $[0,1]$ into $X$. By construction, for any two dyadic rationals $t, t'$, we have 
 \begin{equation}
 \label{eq: midpoint Lipschitz}
 d(\gamma(t),\gamma(t')) = |t-t'|d(x,z)\,,
 \end{equation}
 i.e., $\gamma$ is Lipschitz. In particular, $\gamma$ is then uniformly continuous and we can apply Theorem \ref{thm: extension theorem} as the dyadic rationals form a dense subset of $[0,1]$. Thus, we obtain a curve from $x$ to $z$, and $d(x,z)=L_d(\gamma)$ follows by \eqref{eq: midpoint Lipschitz} as well. 

 (ii) can be shown by slightly modifying the proof of (i).
\end{proof}

\section{Results}
We now turn to the Lorentzian equivalence between the existence of (approximate) $\tau$-midpoints and being (strictly) intrinsic: for the one direction, note that Proposition \ref{pop: intrinsic has midpoints} can be transferred with ease. Conerning the other direction, however, we quickly notice some obstacles. While constructing a dyadic curve $\gamma$ consisting of $\tau$-midpoints poses no issues, an inequality in the spirit of \eqref{eq: midpoint Lipschitz} for $\tau$, i.e., 
\begin{equation}
 \label{eq: tau midpoints property}
 \tau(\gamma(t), \gamma(t')) = |t-t'|\tau(x,z) \, ,
\end{equation}
does not yield anything in relation to the continuity of said map $\gamma$. Roughly speaking, this is because $d$ is directly involved with how close points are in a topological sense, while $\tau$ has no influence in this regard at all. 
Even if we assume $\tau$ to be continuous, we can only infer from points having small $d$-value, i.e., being topologically close, that their $\tau$-value is also small, i.e., that little time passes between the two events. 
Any implication in the other direction seems almost paradoxical: even in Minkowski space, points can have arbitrarily small $\tau$-distance from each other while having arbitrarily large $d$-distance, by choosing points which are almost null related but very far apart. 
In some sense, this wraps around to the original problem described at the start of the previous section: on a spacetime, the Lorentzian metric or the time separation function does not give any information about the topological relationship between two events. 
\medskip

In summary, we need some property which replaces this unwanted implication in the other direction (see Theorem \ref{thm: midpoints imply strictly intrinsic} below). 
It turns out that requiring some control about the time function values of midpoints in comparison to the time function values of the endpoints, i.e., $\tau$-midpoints are almost midpoints with respect to the time function, solves the issue. 
Before this, however, we show the easier direction concerning the existence of (approximate) midpoints in a (strictly) intrinsic \LpLSn. 
\begin{pop}[Being intrinsic implies existence of midpoints]
Let $X$ be a chronological \LpLS with $\tau$ continuous near the diagonal. 
\begin{itemize}
 \item[(i)] If $X$ is geodesic, then $X$ has $\tau$-midpoints. 
 \item[(ii)] If $X$ is intrinsic, then $X$ has approximate $\tau$-midpoints. 
\end{itemize}
\end{pop}
\begin{proof}
 First, note that by \cite[Lemma 3.33]{KS18}, the map $\phi: t \mapsto L_{\tau}(\gamma|_{[a,t]})$ is continuous in $t$ for any causal curve $\gamma:[a,b] \to X$.

 (i) Let $x \ll z$ and let $\gamma:[a,b] \to X$ be a distance realizer from $x$ to $z$. By continuity of $\phi$, we find $s \in [a,b]$ such that $L_{\tau}(\gamma|_{[a,s]})=\frac{1}{2}L_{\tau}(\gamma)=\frac{1}{2}\tau(x,z)$. The claim then follows since the restriction of a distance realizer is still a distance realizer, i.e., setting $y=\gamma(s)$, we have $\tau(x,y)=L_{\tau}(\gamma|_{a,s]})=\frac{1}{2}\tau(x,z)$ and similarly for $\tau(y,z)$. 

 (ii) Let $x \ll z$. Given $\varepsilon >0$, let $\gamma : [a,b] \to X$ be a causal curve from $x$ to $z$ such that $L_{\tau}(\gamma) > \tau(x,z) - \varepsilon$. By continuity of $\phi$, we find $s \in [a,b]$ such that $L_{\tau}(\gamma|_{[a,s]})=\frac{1}{2}L_{\tau}(\gamma) > \frac{1}{2}\tau(x,z) - \frac{\varepsilon}{2}$. Let $y=\gamma(s)$, then $\tau(x,y) \geq L_{\tau}(\gamma|_{[a,s]}) > \frac{1}{2}\tau(x,z) - \frac{\varepsilon}{2}$. Similarly, we obtain $\tau(y,z) > \frac{1}{2}\tau(x,z) - \frac{\varepsilon}{2}$. 
 Conversely, by the reverse triangle inequality we have $\tau(x,y) \leq \tau(x,z) - \tau(y,z) < \tau(x,z) - \frac{1}{2}\tau(x,z) + \frac{\varepsilon}{2} = \frac{1}{2}\tau(x,z) + \frac{\varepsilon}{2}$. Analogously, $\tau(y,z) < \frac{1}{2}\tau(x,z) + \frac{\varepsilon}{2}$ and the proof is finished. 
\end{proof}

\begin{thm}[$\tau$-midpoints imply being causally geodesic]
\label{thm: midpoints imply strictly intrinsic}
Let $X$ be a locally causally closed and sufficiently causally connected \LpLSn. Let $T: X \to \mb{R}$ be a locally anti-Lipschitz time function and suppose that $X$ is equipped with the null-distance $d_T$ with respect to $T$. Assume that $(X,d_T)$ is complete and that $X$ has $\tau$-midpoints. Assume further that there exists $c \in (0,\frac{1}{2}]$ such that the time function is compatible with midpoints in the following sense: for any $p, q \in X, p \ll q$ there is a $\tau$-midpoint $m$ of $p$ and $q$ such that 
\begin{equation}
\label{eq: time function midpoint compatible}
(1-c) T(p) + c T(q) \leq T(m) \leq c T(p) + (1-c) T(q) \, . 
\end{equation}
Then $X$ is causally geodesic.
\end{thm}

We try to give some meaning to \eqref{eq: time function midpoint compatible}: this is the compatibility between $\tau$-midpoints and the time function which we mentioned above. 
The value $c$ controls how far off a $\tau$-midpoint is allowed to be from a midpoint with respect to the time function.

In the worst case of $c=0$, which is ruled out by assumption, the defining equation is trivial, as $T(p) \leq T(m) \leq T(q)$ has to be satisfied for causally related triples for any time function by definition. 
The other extreme, $c=\frac{1}{2}$, demands that some $\tau$-midpoint is precisely a midpoint with respect to $T$, i.e., \eqref{eq: time function midpoint compatible} forces $T(m) = \frac{1}{2}( T(p) + T(q))$. 
In summary, we require there to be \emph{some} control for a $\tau$-midpoint by the time function. 
We do not necessarily care how good the control is, just for there to be some (uniform) control. 
In Remark \ref{rem: visual explanation}, we will demonstrate how \eqref{eq: time function midpoint compatible} (and the corresponding \eqref{eq: time function midpoint compatible2} for $\varepsilon$-$\tau$-midpoints) can be visualized in Minkowski space. 

\begin{proof}
If $p\leq q$ are null related, causal path connectedness implies the existence of a null curve between them, which is automatically a null realizer.

So let $p\ll q$. As in the metric proof, we inductively define a function $\gamma$ on the dyadic rationals in $[0,1]$ as follows, using only $\tau$-midpoints which satisfy \eqref{eq: time function midpoint compatible}: let $\gamma(0)=p, \gamma(1)=q$ and $\gamma(\frac{1}{2})=m$ for a $\tau$-midpoint $m$ of $p$ and $q$. 
Further, define $\gamma(\frac{1}{4})$ to be a $\tau$-midpoint of $p$ and $m$ and define $\gamma(\frac{3}{4})$ to be a $\tau$-midpoint of $m$ and $q$. Continuing iteratively, we obtain that $\gamma(\frac{2k+1}{2^{n+1}})$ is a $\tau$-midpoint of $\gamma(\frac{k}{2^n})$ and $\gamma(\frac{k+1}{2^n})$. 
In this way, we end up with a dyadic ``curve'' $\gamma:[0,1]\cap\left\{\frac{k}{2^n}\colon k,n\in\mb{N}\right\}\to X$. Note that by \eqref{eq: tau midpoints property}, for any two dyadic rationals $t,t'$ we have 
\begin{equation}
\label{eq: dyadic curve timelike strictly intrinsic}
 t < t' \iff \gamma(t) \ll \gamma(t') \, .
\end{equation}
We claim that $\gamma$ is $\alpha$-Hölder, where $\alpha \coloneqq -\log_2(1-c)$. 
This then implies that $\gamma$ is uniformly continuous and we can apply Theorem \ref{thm: extension theorem}.
Towards this claim, first note that \eqref{eq: time function midpoint compatible} implies, after reformulating, that for a midpoint $y$ of $x\ll z$, it holds that
\begin{align*}
c(T(z)-T(x))&\leq T(y)-T(x)\leq (1-c)(T(z)-T(x)) \, , \\
c(T(z)-T(x))&\leq T(z)-T(y)\leq (1-c)(T(z)-T(x)) \, .
\end{align*}
Applying the inequalities on the right iteratively to points on $\gamma$, we see that 
\begin{align}
\label{taumidptsintr:eqSubseqDyad}
& d_{T}\left(\gamma(\frac{k}{2^n}),\gamma(\frac{k+1}{2^n})\right) = T\left(\gamma\left(\frac{k+1}{2^n}\right)\right)-T\left(\gamma\left(\frac{k}{2^n}\right)\right) \\
&\leq (1-c)\left(T\left(\gamma\left(\frac{\lfloor k/2 \rfloor +1}{2^{n-1}}\right)\right)-T\left(\gamma\left(\frac{\lfloor k/2 \rfloor}{2^{n-1}}\right)\right)\right)\nonumber \leq \ldots \\
&\leq (1-c)^{n-1}\left(T\left(\gamma\left(\frac{\lfloor k/2^{n-1} \rfloor +1}{2}\right)\right)-T\left(\gamma\left(\frac{\lfloor k/2^{n-1} \rfloor }{2}\right)\right)\right)\nonumber \\
&\leq (1-c)^n(T(q)-T(p)) \, ,\nonumber 
\end{align}
where the equality in the first line holds by \eqref{eq: null distance between causal points}. 
Intuitively, this computation can be described in words as follows: the interval $[\frac{k}{2^n},\frac{k+1}{2^n}]$ of subsequent midpoints can be reached by halving the original interval and choosing the correct side each time. 
We essentially reverse this process and apply the appropriate reformulation of \eqref{eq: time function midpoint compatible} in each step.

Now consider any two dyadic rationals and express them with the same denominator, i.e., consider $\frac{k_1}{2^n}<\frac{k_2}{2^n}$. Denote by $\frac{l}{2^N}$ the unique dyadic rational with the lowest denominator in $[\frac{k_1}{2^n},\frac{k_2}{2^n}]$. Then $\frac{k_2}{2^n}-\frac{l}{2^N} \in [0,\frac{1}{2^N})$ is a dyadic rational with denominator at most $2^n$. We bring it to the form 
\begin{equation}
\frac{k_2}{2^n}-\frac{l}{2^N} = \sum_{i=N+1}^{n} c_i\frac{1}{2^i} \, ,
\end{equation}
with $c_i \in \{0,1\}$. That is, $c_{N+1}, c_{N+2}, \cdots, c_n$ are the coefficients (of descending rank) in the binary representation of $k_2-l2^{n-N}$. We set $d_j=\frac{l}{2^N} + \sum_{i=N+1}^{j} c_i\frac{1}{2^i}$, i.e., the numbers obtained by cutting off the binary representation of $\frac{k_2}{2^n}$ after $j$ decimal places. Then $d_N=\frac{l}{2^N}$ and $d_n=\frac{k_2}{2^n}$. 
In particular, $d_{j+1}-d_{j}$ is either $0$ or $\frac{1}{2^{j+1}}$, and both have denominator at most $\frac{1}{2^{j+1}}$. In other words, they are equal or subsequent dyadic rationals, and we can apply \eqref{taumidptsintr:eqSubseqDyad}.
We set $b:= \min\{i:c_i\neq0\}\geq N+1$, so $\frac{1}{2^b}\leq\frac{k_2}{2^n}-\frac{l}{2^N}\leq \frac{1}{2^{b-1}}$. Then we obtain 
\begin{align*}
d(\gamma(\frac{l}{2^N}),\gamma(\frac{k_2}{2^n}))&\leq \sum_{j=b}^{n-1} d(\gamma(d_j),\gamma(d_{j+1}) \leq \sum_{j=b}^{n-1} (1-c)^{j+1}(T(q)-T(p)) \\
&= (1-c)^{b}\frac{1-(1-c)^{n-b}}{1-(1-c)} (T(q)-T(p))
\\
&\leq (1-c)^{b}\frac{1}{c} (T(q)-T(p)) \, ,
\end{align*}
using the geometric sum formula. 
In summary, we used the triangle inequality to sub-divide the distance between arbitrary dyadic rationals into smaller distances of dyadic rationals, which are subsequent on coarseness levels getting finer according to the binary representation of the enumerators, and apply the process of \eqref{taumidptsintr:eqSubseqDyad} to each of these distances. 
As $\frac{1}{2^b}\leq\frac{k_2}{2^n}-\frac{l}{2^N}$, we have $(1-c)^b = \left(\frac{1}{2^b}\right)^{-\log_2(1-c)}\leq \left(\frac{k_2}{2^n}-\frac{l}{2^N}\right)^{-\log_2(1-c)}$. Recalling that $\alpha=-\log_2(1-c)$, the previous calculation simplifies to 
\begin{equation}
d(\gamma(\frac{l}{2^N}),\gamma(\frac{k_2}{2^n})) \leq \frac{T(q)-T(p)}{c} \left(\frac{k_2}{2^n}-\frac{l}{2^N}\right)^\alpha \, .
\end{equation}
Analogously, we obtain 
\begin{equation}
d(\gamma(\frac{k_1}{2^n}),\gamma(\frac{l}{2^N})) \leq \frac{T(q)-T(p)}{c} \left(\frac{l}{2^N}-\frac{k_1}{2^n}\right)^\alpha \, .
\end{equation}
Putting these together, we have 
\begin{align*}
d(\gamma(\frac{k_1}{2^n}),\gamma(\frac{k_2}{2^n})) & \leq 
d(\gamma(\frac{k_1}{2^n}),\gamma(\frac{l}{2^N})) + d(\gamma(\frac{l}{2^N}),\gamma(\frac{k_2}{2^n})) \\ 
& \leq \frac{T(q)-T(p)}{c}\left(\left(\frac{k_2}{2^n}-\frac{l}{2^N}\right)^{\alpha}+ \left(\frac{l}{2^N}-\frac{k_1}{2^n}\right)^{\alpha}\right) \\ 
& \leq 
2\frac{T(q)-T(p)}{c}\left(\left(\frac{k_2}{2^n}-\frac{k_1}{2^n}\right)\right)^\alpha \, ,
\end{align*}
where we used the elementary identity $a^{\alpha} + b^{\alpha} \leq 2(a+b)^{\alpha}$ for $a,b\geq0, \alpha >0$. 
This shows that $\gamma$ is $\alpha$-Hölder. 
Thus, we can apply Theorem \ref{thm: extension theorem} to obtain a continuous extension of $\gamma$ to $[0,1]$, which we shall denote again by $\gamma$. 
\medskip

It is left to show that $\gamma : [0,1] \to X$ is causal and distance realizing. 
Concerning causality, for any two real numbers $t_1<t_2$ in $[0,1]$, we need to show that $\gamma(t_1)\leq\gamma(t_2)$. First, let $t_1=\frac{k_1}{2^n},t_2=\frac{k_2}{2^n}$ be dyadic rationals, then this follows immediately by \eqref{eq: dyadic curve timelike strictly intrinsic}. 
Now let $t_1<t_2$ be arbitrary. Let $U_i$ be an open causally closed neighbourhood of $\gamma(t_i),\, i=1,2$, and find dyadic rationals $t_1<\frac{k_1}{2^n}<\frac{k_2}{2^n}< t_2$ such that $\gamma(\frac{k_i}{2^n})\in U_i$. 
For a sequence $s_n\to t_1$ of dyadic rationals, we have $s_n \leq \frac{k_1}{2^n}$ for almost all $n$, so $\gamma(s_n)\leq\gamma(\frac{k_1}{2^n})$. By the causal closedness of $U_i$, we infer $\gamma(t_1)\leq\gamma(\frac{k_1}{2^n})$, and similarly $\gamma(\frac{k_2}{2^n})\leq\gamma(t_2)$. 
In total, $\gamma(t_1)\leq\gamma(\frac{k_1}{2^n}) \ll \gamma(\frac{k_2}{2^n})\leq\gamma(t_2)$. That is, $\gamma$ is even a timelike curve.

Finally, we show that $\gamma$ is a distance realizer with constant speed $L:=\tau(p,q)$: For $t_1<t_2$, we need to show that $\tau(\gamma(t_1),\gamma(t_2))=L(t_2-t_1)$. First, let $t_1=\frac{k_1}{2^n},t_2=\frac{k_2}{2^n}$ be dyadic rationals. Then this follows by induction on $n$ (noting that these are constructed as midpoints). 
For general $t_1,t_2 \in [0,1]$, this now follows by reverse triangle inequality of $\tau$. 
Indeed, let $(\frac{k_1^{(n)}}{2^n})_n, (\frac{k_1^{(n)}+1}{2^n})_n$ and $(\frac{k_2^{(n)}}{2^n})_n, (\frac{k_2^{(n)}+1}{2^n})_n$ be sequences of dyadic rationals approximating $t_1$ and $t_2$ from above and below, respectively, where the $(n)$ symbolizes that the enumerators are dependent on $n$. Then we have 
\begin{equation}
\label{eq: str intrinsic distance realizer}
\underbrace{\tau(\gamma(\frac{k_1^{(n)}+1}{2^n}),\gamma(\frac{k_2^{(n)}}{2^n}))}_{=\frac{k_2^{(n)}-k_1^{(n)}}{2^n}L -\frac{1}{2^n}L}\leq\tau(\gamma(t_1),\gamma(t_2))\leq\underbrace{\tau(\gamma(\frac{k_1^{(n)}}{2^n}),\gamma(\frac{k_2^{(n)}+1}{2^n}))}_{=\frac{k_2^{(n)}-k_1^{(n)}}{2^n}L +\frac{1}{2^n}L} \, .
\end{equation}
For $n\to+\infty$, it follows that $\frac{k_1^{(n)}}{2^n}\to t_1$ and $\frac{k_2^{(n)}}{2^n}\to t_2$, so the distance realizer property also holds for these parameters.
\end{proof}
This result can also be generalized for $\varepsilon$-$\tau$-midpoints, giving that $X$ is intrinsic:
\begin{thm}[Approximate $\tau$-midpoints imply being intrinsic]
\label{thm: approx midpoints imply intrinsic}
Let $X$ be a locally causally closed and sufficiently causally connected \LpLSn. Let $T: X \to \mb{R}$ be a time function and suppose that $X$ is equipped with the null-distance $d_T$ with respect to $T$. Assume that $(X,d_T)$ is complete and that $X$ has approximate $\tau$-midpoints. Assume further that there exists $c \in (0,\frac{1}{2}]$ and $\hat{\varepsilon}>0$ such that the time function is compatible with approximate $\tau$-midpoints in the following sense: for any $p, q \in X, p \ll q$ there is a $\tau(p,q)\hat{\varepsilon}$-$\tau$-midpoint $m$ of $p$ and $q$ such that 
\begin{equation}
\label{eq: time function midpoint compatible2}
(1-c) T(p) + c T(q) \leq T(m) \leq c T(p) + (1-c) T(q) \, .
\end{equation}
Then $X$ is intrinsic.
\end{thm}

\begin{rem}[A visual presentation for \eqref{eq: time function midpoint compatible2}]\label{rem: visual explanation}
Now we try to give some visual understanding of this relationship. To this end, first note that the set of $\varepsilon$-$\tau$-midpoints in a \LpLS $X$ can be described as follows: let $\varepsilon>0$ and let $p \ll q$. 
Let $H_p^{\varepsilon, +}=\{y \in X \mid \tau(p,y) > \frac{1}{2}\tau(p,q) - \varepsilon\}$ and $H_q^{\varepsilon, -}=\{y \in X \mid \tau(y,q) > \frac{1}{2}\tau(p,q) - \varepsilon\}$, then $L=H_p^{\varepsilon, +} \cap H_q^{\varepsilon, -}$ is the set of $\varepsilon$-$\tau$-midpoints of $p$ and $q$. 
Also define $S=\{y\in X \mid (1-c)T(p)+cT(q)\leq T(y)\leq cT(p)+(1-c)T(q)\}$. 
Then \eqref{eq: time function midpoint compatible2} requires that $L\cap S$ is non-empty. 

In (1+1)-dimensional Minkowski space $\mb{R}^{1,1}$, the sets $H_p^{\varepsilon}$ and $H_q^{\varepsilon}$ are confined by a hyperbola centered at $p$ and $q$, respectively, and $L$ is an open set which might visually be desribed as a ``lens''. 
With the canonical time function denoted by $T$, the set $S$ is a horizontal strip with width $(1-2c)\tau(p,q)$ symmetric around the horizontal line $\frac{1}{2}(T(q)-T(p))$. 
In Minkowski space, elementary calculations show that by choosing $\varepsilon$ small enough one can even ensure that $L\subseteq S$, see Figure \ref{fig: lens and band}. 
In particular, \eqref{eq: time function midpoint compatible2} holds. 
\begin{figure}
\begin{center}
\begin{tikzpicture}

\draw [samples=50,domain=-0.46:0.46,rotate around={90:(4,2.5)},xshift=4cm,yshift=2.5cm] plot ({1.3*(1+(\x)^2)/(1-(\x)^2)},{1.3*2*(\x)/(1-(\x)^2)});
\draw [samples=50,domain=-0.46:0.46,rotate around={90:(4,6.5)},xshift=4cm,yshift=6.5cm] plot ({1.3*(-1-(\x)^2)/(1-(\x)^2)},{1.3*(-2)*(\x)/(1-(\x)^2)});
\draw [domain=-1.1558404941238893:7.349471484124196] plot(\x,{(--3.7-0*\x)/1});
\draw [domain=-1.1558404941238893:7.349471484124196] plot(\x,{(--5.3-0*\x)/1});
\draw [dashed] (0,7.5)-- (0,2);
\begin{scriptsize}
\coordinate [circle, fill=black, inner sep=0.5pt, label=270: {$p$}] (p) at (4,2.5);
\coordinate [circle, fill=black, inner sep=0.5pt, label=90: {$q$}] (q) at (4,6.5);
\coordinate [circle, fill=black, inner sep=0.5pt, label=180: {$m$}] (m) at (4,4.5);

\draw (4.5,4.5) node {$L$};
\draw (0.2,7.5) node {$T$};
\draw (8.8,3.3) node {$(1-c)T(p)+c T(q)$};
\draw (8.8,5.7) node {$c T(p) + (1-c)T(q)$};
\end{scriptsize}
\end{tikzpicture}
\caption{A lens contained in a strip, with $m$ being a $\tau$-midpoint of $p$ and $q$.}
\label{fig: lens and band}
\end{center}
\end{figure}
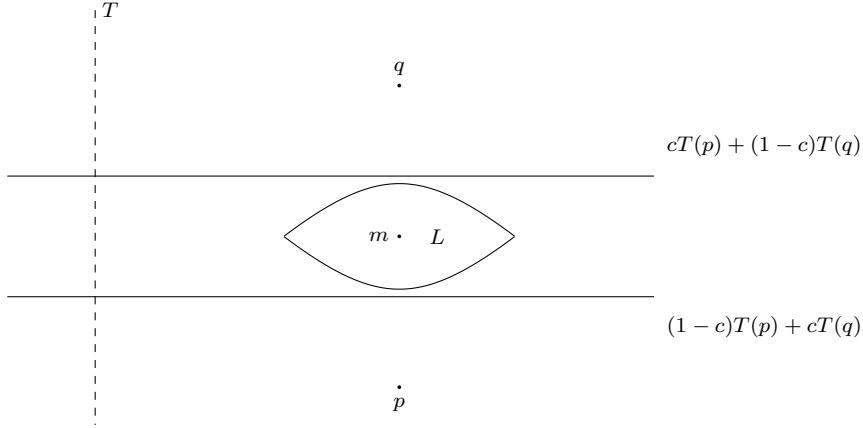
\end{rem}

\begin{proof}
In general, this proof is very similar to the one of Theorem \ref{thm: midpoints imply strictly intrinsic}. We highlight the technical differences and in turn shorten some of the unchanged parts. 
In complete analogy to the previous proof, note that if $p \leq q$ are null related, then by causal path connectedness there exists a causal curve joining them which necessarily realizes the distance. 

So let $p\ll q$ and let $\varepsilon\in(0,\hat{\varepsilon}\tau(p,q))$. 
We inductively define a function $\gamma$ on the dyadic rationals in $[0,1]$ as follows, using only $\varepsilon$-$\tau$-midpoints which satisfy \eqref{eq: time function midpoint compatible2}: let $\gamma(0)=p, \gamma(1)=q$ and $\gamma(\frac{1}{2})=m$, where $m$ is an $\frac{\varepsilon}{4}$-$\tau$-midpoint of $x$ and $z$. 
Then define $\gamma(\frac{1}{4})$ to be a $\frac{\varepsilon}{16}$-$\tau$-midpoint of $p$ and $m$ and define $\gamma(\frac{3}{4})$ to be a $\frac{\varepsilon}{16}$-$\tau$-midpoint of $m$ and $q$. 
Furthermore, let $\gamma\left(\frac{k}{2^n}\right)$ be inductively defined for all $k$, then set $\gamma\left(\frac{2k+1}{2^{n+1}}\right)$ to be a $\frac{\varepsilon}{4^{n+1}}$-$\tau$-midpoint between $\gamma\left(\frac{k}{2^n}\right)$ and $\gamma\left(\frac{k+1}{2^n}\right)$. 
In this way, we end up with a dyadic ``curve'' $\gamma:[0,1]\cap\{\frac{k}{2^n}\colon k,n\in\mb{N}\}\to X$. 
First, note that by reforming \eqref{eq: epsilon-midpoints definition}, we get 
\begin{equation}
\left| \frac{\tau(p,m)}{\tau(p,q)}-\frac{1}{2} \right|<\frac{\varepsilon}{4 \tau(p,q)} \quad\text{and}\quad \left| \frac{\tau(m,q)}{\tau(p,q)}-\frac{1}{2} \right|<\frac{\varepsilon}{4 \tau(p,q)} \, .
\end{equation}
Similarly, as $\gamma(\frac{k}{2^n})$ or $\gamma(\frac{k+1}{2^n})$ (depending on the parity of $k$) is a $\frac{\varepsilon}{4^n}$-midpoint of $\gamma(\frac{[k/2]}{2^{n-1}})$ and $\gamma(\frac{[k/2]+1}{2^{n-1}})$, we have 
\begin{equation}
\label{eq: induction step intrinsic}
\tau\left(\gamma(\frac{k}{2^n}),\gamma(\frac{k+1}{2^n})\right)>\left(\frac{1}{2}\tau\left(\gamma(\frac{[k/2]}{2^{n-1}}),\gamma(\frac{[k/2]+1}{2^{n-1}})\right)-\frac{\varepsilon}{4^n}\right) \, . 
\end{equation}
By induction, we finally arrive at 
\begin{equation}\label{eq: dyadic curve timelike intrinsic}
\tau\left(\gamma(\frac{k}{2^n}),\gamma(\frac{k+1}{2^n})\right)>
\frac{1}{2^n}\tau(p,q)-\underbrace{\sum_{i=0}^{n-1}\frac{\varepsilon}{4^{n-i}\cdot 2^i}}_{=\frac{\varepsilon}{2^n}\sum_{i=1}^n\frac{1}{2^i}}>\frac{1}{2^n}(\tau(p,q)-\varepsilon) \, . 
\end{equation}
In particular, for any dyadic rationals $t,t'$ we conclude 
\begin{equation}
\label{eq: dyadic curve timelike relation}
t<t' \iff \gamma(t)\ll\gamma(t') \, .
\end{equation} 
As before, we claim that $\gamma$ is $\alpha$-Hölder, where $\alpha:=-\log_2(1-c)$. 
Note that we obtain the same inequalities from \eqref{eq: time function midpoint compatible2} as in the previous proof, that is, for a $\hat{\varepsilon}\tau(x,z)$-$\tau$-midpoint $y$ of $x\ll z$, we have 
\begin{align*}
c(T(z)-T(x))&\leq T(y)-T(x)\leq (1-c)(T(z)-T(x)) \, , \\
c(T(z)-T(x))&\leq T(z)-T(y)\leq (1-c)(T(z)-T(x)) \, .
\end{align*}
The next step in the proof of Theorem \ref{thm: midpoints imply strictly intrinsic} relies only on the compatibility assumption of $\tau$ and $T$ and properties of $d_T$, i.e., this can be copied as is. Hence we obtain the same inequality as in \eqref{taumidptsintr:eqSubseqDyad}:
\begin{equation}
d_{T}\left(\gamma(\frac{k}{2^n}),\gamma(\frac{k+1}{2^n})\right) = \ldots 
\leq (1-c)^n(T(q)-T(p)) \label{taumidptsintr:eqSubseqDyad2} \, .
\end{equation}
Moreover, up to and including the fact that $\gamma$ is $\alpha$-Hölder, the proof is completely analogous. 
Hence, we can again apply Theorem \ref{thm: extension theorem} to obtain a continuous extension of $\gamma$ to $[0,1]$, which we shall denote by $\gamma$ as well. 
\medskip

It is left to show that the extension $\gamma : [0,1] \to X$ is still causal and $L_{\tau}(\gamma) + \varepsilon > \tau(p,q)$. 
As in the previous proof of Theorem \ref{thm: midpoints imply strictly intrinsic}, we infer that $\gamma$ is even a timelike curve.

Finally, we show that $\gamma$ has speed at least $L:=\tau(p,q)-\varepsilon$, i.e., for $t_1<t_2$, we need to show that 
\begin{equation}
 \label{eq: almost midpoints speed}
 \tau(\gamma(t_1),\gamma(t_2)) \geq L(t_2-t_1) \, . 
\end{equation}
By \eqref{eq: dyadic curve timelike intrinsic}, we have already shown \eqref{eq: almost midpoints speed} for subsequent dyadic rationals $t_1=\frac{k}{2^n},\,t_2=\frac{k+1}{2^n}$ (notice that $\frac{k+1}{2^n}-\frac{k}{2^n}=\frac{1}{2^n}$). 

Now let $t_1=\frac{k_1}{2^n},t_2=\frac{k_2}{2^n}$ be any dyadic rationals, then this follows by reverse triangle inequality, i.e., we have
\begin{equation}
\tau(\gamma(\frac{k_1}{2^n}),\gamma(\frac{k_2}{2^n}))\geq \sum_{i=k_1}^{k_2-2} \tau(\gamma(\frac{i}{2^n}),\gamma(\frac{i+1}{2^n}))\geq L(\frac{k_2}{2^n}-\frac{k_1}{2^n}) \, .
\end{equation}
For general $t_1,t_2$, this also follows by reverse triangle inequality of $\tau$, see the lower inequality in \eqref{eq: str intrinsic distance realizer}. 
The upper inequality can be obtained by flipping the inequality in \eqref{eq: induction step intrinsic} and the sign in front of $\varepsilon$ and following the proof again. 
That is, $\gamma$ is up to a factor of $1+\frac{\varepsilon}{\tau(p,q)}$ in constant speed parametrization. 
\end{proof}
\begin{rem}[Local weak causal closure]
\label{rem: on loc caus closure}
 Regarding the formulation of local weak causal closure found in \cite[Definition 2.16]{ACS20}, this seems more suitable for our setting, since we are causally path connected, but we are not (necessarily) in a strongly causal setting. 
 Clearly, local causal closure is only used when showing that $\gamma$ is a causal curve. 
 However, there might not exist a sequence of dyadic rationals $s_n \to t_1$ such that $\gamma(s_n) \leq_U \gamma(\frac{k_1}{2^n})$, i.e., while we know $\gamma(s_n) \leq \gamma(\frac{k_1}{2^n})$, we do not know whether the connecting causal curve lies inside $U$.
\end{rem}
\begin{chapt*}{Acknowledgments}
We acknowledge the kind hospitality of the Erwin Schr\"odinger International Institute for Mathematics and Physics (ESI) during the workshop \emph{Nonregular Spacetime Geometry}, where parts of this research were carried out. 

This work was supported by research grant P33594 of the Austrian Science Fund FWF.
\end{chapt*}
\bibliographystyle{alpha}
\bibliography{references}
\addcontentsline{toc}{section}{References}

\end{document}